\documentclass[12pt]{amsart}
\usepackage{amsfonts,amssymb,amsmath,amsthm}
\usepackage{enumerate}
\usepackage{bbm}
\usepackage{times}
\numberwithin{equation}{section}

\newtheorem{theor}{Theorem}[section]

\newcounter{other}            % Questions get letters
\def\B{\mathcal{B}}
\def\D{\mathbb{D}}
\def\C{\mathbb{C}}

\begin{document}

\title
{The Bloch space and the dual space of a Luecking-type subspace of $A^1$}
\author{Guanlong Bao and Fangqin Ye}
\address{Guanlong Bao\\
Department of Mathematics\\
    Shantou University\\
    Shantou, Guangdong 515063, China}
\email{glbao@stu.edu.cn}

\address{Fangqin Ye\\
Business School\\
    Shantou University\\
    Shantou, Guangdong 515063, China}
\email{fqye@stu.edu.cn}

\thanks{The work  was   supported  by NNSF of China (No. 11371234 and No. 11571217).}

\begin{abstract}
Let $X$ be the dual space of a Luecking-type subspace of the Bergman space $A^1$.
It is known that the Bloch space $\mathcal{B}$ is a subset of $X$. In 1990,  Ghatage and  Sun asked whether $\mathcal{B}$ is dense in $X$. They also asked whether the little version of $X$ is a subset of $\mathcal{B}$. In this note, based on results and methods  of  Girela,  Pel\'aez, P\'erez-Gonz\'alez and  R\"atty\"a in 2008, we answer the  two questions in the negative.

\end{abstract}

\maketitle

\section{Introduction}

Let $\D$ be the open unit disk in the complex plan $\C$. Denote by $H(\D)$ the space of analytic functions in $\D$. For $0< p<\infty$,
a function $f\in H(\D)$ belongs to the Bergman
space $A^p$ if
$$
\|f\|_{A^p}=\left(\int_\D |f(z)|^p dA(z)\right)^{\frac{1}{p}}<\infty,
$$
where $dA$ is the normalized area measure of $\D$.  Coifman and  Rochberg \cite{CR} established  atomic decomposition theorem for the Bergman space  $A^p$, $0<p<\infty$.
In \cite[p. 329]{Luecking}, Luecking proved that there exists a sequence $\{a_n\}_{n=1}^\infty\subseteq \D$ such that every $f\in A^p$, $p>1$, can be written in the form
$$
f(z)=\sum_{n=1}^\infty \frac{c_n (1-|a_n|)^{(1-1/p)+\eta}}{(1-\overline{a_n}z)^{1+1/p+\eta}}
$$
for an arbitrary $\eta>0$ and $\{c_n\}\in \ell^p$. As pointed out in \cite{Luecking},  the range of the parameter $\eta$ in the above formula is larger than the corresponding result   in \cite{CR}.
To determine the extent to which Luecking's decomposition of Bergman spaces $A^p$, $1<p<\infty$,  can be extended to $A^1$, Ghatage and Sun \cite{GS2} introduced  a Banach space of analytic functions, denoted by $Y$. More precisely, $Y$ is the completion (in the norm $\|.\|_Y$ defined later) of the set  of functions $f\in A^1$ with the form
$$
f(z)=\sum_{n=1}^\infty a_n k_{\lambda_n}(z),
$$
where $a_n\in \C$, $\lambda_n\in \D$ and
$$
\sum_{n=1}^\infty|a_n| \|k_{\lambda_n}\|_{A^1} <\infty.
$$
Here
$$
k_w(z)=\frac{1}{(1-\overline{w}z)^2}, \ w\in \D,
$$
is the  reproducing kernel of $A^2$. A norm of $Y$ is
$$
\|f\|_Y= \inf \sum_{n=1}^\infty|a_n| \|k_{\lambda_n}\|_{A^1},
$$
where the infimum is taken over the set of all such decompositions of $f$.  From \cite{GS2}, $Y$ is a proper subset of $A^1$ and  $Y$ is said to be  a Luecking-type subspace of  $A^1$.

 Ghatage and  Sun \cite{GS2} gave a series of interesting results of the space $Y$.  In particular, they described its dual and predual. Let $X$ be the space of functions $f\in H(\D)$ with
$$
\|f\|_X=\sup_{z\in \D} |f(z)|\|k_z\|^{-1}_{A^1}<\infty.
$$
Denote by $X_0$ the little version of $X$. Namely, the space $X_0$ consists of functions $f\in X$ such that
$$
\lim_{|z|\rightarrow 1}|f(z)|\|k_z\|^{-1}_{A^1}=0.
$$
It is known that $X_0$ is a closed subspace of $X$.  Ghatage and  Sun \cite{GS2} proved that the dual of $X_0$ can be identified with $Y$, and the dual of $Y$ can be
identified with $X$. Recall that the Bloch space $\B$ is the set of functions $f\in H(\D)$ for which
$$
\sup_{z\in \D} (1-|z|^2)|f'(z)|<\infty.
$$
The little Bloch space $\B_0$ consists of those functions $f\in \B$ with
$$
\lim_{|z|\rightarrow 1} (1-|z|^2)|f'(z)|=0.
$$
By \cite{GS2},  $\B \subsetneqq X$ and $\B_0 \subsetneqq X_0$. See \cite{CEGY, Choe-Lee, GS1, Lee} for the further study associated with   $X$ or $Y$.

In \cite[p. 771]{GS2},  Ghatage and  Sun asked whether $\mathcal{B}$ is dense in $X$.
In \cite[p. 773]{GS2}, they asked whether $X_0$ is a subset of $\B$.
In this note,  we answer  the  two questions  in the negative.

\section{The Bloch space and the space $X$}

The section is devoted to  answer the two questions  of   Ghatage and Sun. By results and methods of  Girela,   Pel\'aez,  P\'erez-Gonz\'alez and  R\"atty\"a \cite{GPPR}, the  proof given here   is elementary.

Denote by  $ H_{\log}^\infty$  the Banach space of functions $f\in H(\D)$ satisfying
$$
\|f\|_{H_{\log}^\infty}=\sup_{z\in \D}\frac{|f(z)|}{\log\frac{e}{1-|z|}}<\infty.
$$
See \cite{GPPR} for the study of $ H_{\log}^\infty$.
From  \cite{GS1} or  \cite{GS2}, we know that
\begin{equation*}
\|k_{w}\|_{A^1}=
\begin{cases}
1, \ \ w=0,\\
|w|^{-2} \log \frac{1}{(1-|w|^2)} , \ w\neq 0.
\end{cases}
\end{equation*}
Note that
$$
\lim_{|w|\rightarrow 0} |w|^{-2} \log \frac{1}{(1-|w|^2)} =1.
$$
Then it is easy to see that $X=H_{\log}^\infty$. Hence  another  norm of $X$ can be  defined by
$$
|||f|||_X=\sup_{z\in \D}\frac{|f(z)|}{\log\frac{e}{1-|z|}}, \ \ f\in X.
$$
 Also,  $X_0$ is the set of functions $f\in X$ with
 $$
 \lim_{|z|\rightarrow 1}  \frac{|f(z)|}{\log\frac{e}{1-|z|}} =0.
 $$

From \cite{GS2}, $\B_0$ is densely contained in $X_0$.  But the following result shows that  the case for $\B$ and $X$ is different.
\begin{theor}\label{qestion 1}
The Bloch space is not dense in the space $X$.
\end{theor}
\begin{proof}  By \cite[Theorem 1.2]{GPPR}, there exist two functions $f_1$, $f_2\in X$ such that
\begin{equation} \label{f1 f2}
|f_1(z)|+|f_2(z)|\geq \log\frac{1}{1-|z|}, \ z\in \D.
\end{equation}
Suppose  $\B$ is dense in $X$. Then  there are two functions $g_1$, $g_2\in \B$ satisfying
$$
|||f_i-g_i|||_X=\sup_{z\in \D}\frac{|f_i(z)-g_i(z)|}{\log\frac{e}{1-|z|}}<\frac{1}{3}, \ \ i=1, 2.
$$
Hence,
$$
|f_i(w)-g_i(w)|\leq \frac{1}{3} \log\frac{e}{1-|w|},     \ \ i=1, 2,
$$
for all $w\in \D$. Consequently,
$$
|g_i(w)|\geq |f_i(w)|-\frac{1}{3} \log\frac{e}{1-|w|},\ \ w\in \D, \ \ \ \ i=1, 2. $$
Combining this with (\ref{f1 f2}), one gets
$$
|g_1(w)|+|g_2(w)|\geq  \log\frac{1}{1-|w|}-\frac{2}{3} \log\frac{e}{1-|w|},\ \ w\in \D. $$
Thus,
 $$
 |g_1(w)|+|g_2(w)|\geq  \frac{1}{12} \log\frac{e}{1-|w|},\ \  1-e^{-3} <|w|<1,
 $$
which will produce  a contradiction (see \cite[p. 514]{GPPR}). In fact, for $1-e^{-3}<r<1$, we have
$$
\frac{1}{2\pi} \int_0^{2\pi} \left(|g_1(re^{i\theta})|+|g_2(re^{i\theta})|\right) d\theta \geq
\frac{1}{12} \log\frac{e}{1-r}.
$$
Thus,
 $$
 \frac{1}{2\pi} \int_0^{2\pi} |g_j(re^{i\theta})| d\theta \geq     \frac{1}{ 24} \log\frac{e}{1-r}, $$
  where $j$ is equal to   1 or 2.    This gives
   \begin{equation}\label{a Bloch function}
 \varlimsup_{r\rightarrow1} \frac{ \frac{1}{2\pi} \int_0^{2\pi} |g_j(re^{i\theta})| d\theta} {\left(\log\frac{1}{1-r}\right)^{\frac{1}{2} }}=\infty.
 \end{equation}
But a result of Clunie and MacGregor \cite{CM} or  Makarov \cite{Ma} asserts that
\begin{equation} \label{growth of Bloch}
 \left(\frac{1}{2\pi} \int_0^{2\pi} |h(re^{i\theta})|^p d\theta\right )^{\frac{1}{p}}=O\left(\left(\log\frac{1}{1-r}\right)^{\frac{1}{2}}\right),
   \ \ \mbox{as}\ \  r\rightarrow1,
  \end{equation}
for all $0<p<\infty$  and $h\in \B$. We see that condition (\ref{a Bloch function}) contradicts condition (\ref{growth of Bloch}) with $p=1$.   The proof is complete.
\end{proof}

\vspace{0.1truecm}
\noindent{\bf Remark 1}.\quad
The first construction of the same fashion as condition (2.1) was given by Ramey and Ullrich \cite{RU}. More precisely,  Ramey and Ullrich proved that there exist $f$, $g\in \B$ such that
$$
|f'(z)|+|g'(z)|\geq \frac{1}{1-|z|}
$$
for all $z\in \D$. In 2011, Kwon and  Pavlovi\'c \cite{KP} generalized Ramey and Ullrich's result  by considering a wide class of weights.
We refer to \cite{AD, GPR, Lou} for more results related to these   constructions.

\vspace{0.1truecm}
\noindent{\bf Remark 2}.\quad
  Choe and   Lee \cite[p. 162]{Choe-Lee} also  posed a question that is $\B$  dense in $X$ for the corresponding case in the unit ball of the complex space $\C^n$?  Note that the  results in the unit ball of $\C^n$  corresponding to conditions (\ref{f1 f2}) and (\ref{growth of Bloch})  can be found in
\cite[p. 400]{AD} and \cite[p. 2808]{D1} respectively. The same arguments as the proof of Theorem \ref{qestion 1}  yield that the answer to   Choe and
 Lee's question is also negative.

   Girela,   Pel\'aez,  P\'erez-Gonz\'alez and  R\"atty\"a \cite[Theorem 8.1]{GPPR}
characterized certain lacunary series in $H_{\log}^\infty$ (see \cite[Theorem 14]{PR} for a more general result).  We get the   corresponding result   of $X_0$ as follows. As an application,  we show that $X_0\nsubseteq \B$.
\begin{theor}\label{qestion 2}
 Let $f\in H(\D)$ with the power series expansion $f(z)=\sum_{k=0}^\infty a_k z^{n_k}$  and suppose  there exist $\alpha>1$ and $\beta>1$  such that $n_k^\alpha \leq n_{k+1} \leq n_k^{\alpha\beta}$ for all $k$. Then $f\in X_0$  if and only if
      \begin{equation}  \label{X_0 series}
      \lim_{k\rightarrow \infty}    \frac{|a_k|}{\log n_k} =0.
      \end{equation} Furthermore,   the space $X_0$ is not a subset of the Bloch space.
\end{theor}
 \begin{proof}  Let $f\in X_0$. Then
  for any $\varepsilon>0$, there exists a $\delta>0$, such that  $|f(z)|<\varepsilon \log\frac{e}{1-|z|}$ for $\delta<|z|<1$.
  For this $\delta$, there exists a positive integer  $N$, such that if $k>N$, then $1-\frac{1}{n_k}>\delta$. By Cauchy's integral formula,     one gets
  $$
  |a_k|=\frac{|f^{(n_k)}(0)|}{n_k!}\leq \frac{1}{2\pi}  \int_{|z|=1-\frac{1}{n_k}} \frac{|f(w)|}{|w|^{n_k+1}} |dw|
  \leq  \varepsilon \log(en_k)  \left(1-\frac{1}{n_k}\right)^{-n_k} \leq 2 \varepsilon e \log(en_k).
  $$
  Thus condition (\ref{X_0 series})  holds.

  On the other hand, suppose    condition (\ref{X_0 series}) is true. Then  for any $\varepsilon>0$,  there exists a positive integer  $K$, such that if $k>K$, then $|a_k|<\varepsilon \log n_k$.    There also exists a positive constant $\eta$ such that if $\eta<|z|<1$, then
  $$
  \frac{\sum_{k=0}^K |a_k|}{\log \frac{e}{1-|z|}}<\varepsilon.
  $$
  It is clear that     condition (\ref{X_0 series}) implies
  $$
 \sup_{k} \frac{|a_k|}{\log n_k} <\infty.
  $$
Hence, by   Girela,   Pel\'aez,  P\'erez-Gonz\'alez and  R\"atty\"a \cite[p. 536]{GPPR}, we see that   there exists a positive constant $C$ depending only on $\alpha$
and $\beta$, such that
$$
\sum_{k=1}^\infty (\log n_k) r^{n_k} \leq C \log\frac{e}{1-r}, \ \ 0\leq r<1. $$
Therefore, for    $\eta<|z|<1$, we obtain
$$
\frac{|f(z)|}{\log\frac{e}{1-|z|}} \leq    \frac{{\sum_{k=0}^K |a_k|}+\varepsilon \sum_{k=K+1}^\infty (\log n_k)|z|^{n_k}}{\log\frac{e}{1-|z|}}
\leq (1+C) \varepsilon,
$$
 which gives
 $$
 \lim_{|z|\rightarrow 1}   \frac{|f(z)|}{\log\frac{e}{1-|z|}}=0. $$
 Namely  $f\in X_0$.   Consequently, $f\in X_0$ if and only if condition (\ref{X_0 series}) holds.

  Let $h(z)=\sum_{k=1}^\infty k z^{2^{2^k}}$. Take  $\alpha =2$  and $\beta =3/2$. Then
  $(2^{2^{k}})^\alpha=2^{2^{k+1}}< (2^{2^{k}})^{\alpha \beta}$ and
  $$
  \lim_{k\rightarrow \infty} \frac{k}{2^k \log 2}=0. $$
  Thus $h\in X_0$. It is well known (cf. \cite{ACP}) that if $g(z)=\sum_{n=0}^\infty b_n z^n \in \B$,
  then the sequence $\{b_n\}_{n=0}^\infty$ is bounded. Hence $h\not\in \B$. In other words,  the space $X_0$ is not a subset of the Bloch space.
  The proof is complete.
\end{proof}

\bigskip
\noindent
{\bf Acknowledgements}
\bigskip

The authors thank the two anonymous referees very much for their corrections and helpful suggestions.

\end{document}